\documentclass[11pt,a4paper]{article}

\usepackage{inputenc}
\usepackage{amsmath}
\usepackage{bm}
\usepackage{bbold}
\usepackage{amsthm}

\usepackage{hyperref}

\title{An Algebraic Approach\\ to Multidimensional Minimax Location Problems\\ with Chebyshev Distance\thanks{WSEAS Transactions on Mathematics, 2011. Vol.~10, no.~6, pp.~191-200.}
}

\author{Nikolai Krivulin\thanks{Faculty of Mathematics and Mechanics, St.~Petersburg State University, 28 Universitetsky Ave., St.~Petersburg, 198504, Russia, 
nkk@math.spbu.ru}
}

\date{}

\newtheorem{lemma}{Lemma}
\newtheorem{corollary}{Corollary}

\begin{document}

\maketitle

\begin{abstract}
Minimax single facility location problems in multidimensional space with Chebyshev distance are examined within the framework of idempotent algebra. The aim of the study is twofold: first, to give a new algebraic solution to the location problems, and second, to extend the area of application of idempotent algebra. A new algebraic approach based on investigation of extremal properties of eigenvalues for irreducible matrices is developed to solve multidimensional problems that involve minimization of functionals defined on idempotent vector semimodules. Furthermore, an unconstrained location problem is considered and then represented in the idempotent algebra settings. A new algebraic solution is given that reduces the problem to evaluation of the eigenvalue and eigenvectors of an appropriate matrix. Finally, the solution is extended to solve a constrained location problem. 
\\

\end{abstract}

\section{Introduction}

Location problems \cite{Eiselt11Pioneering} form one of the classical research domains in optimization that has its origin dating back to XVIIth century and classical works by P.~Fermat, E.~Torricelli, J.~J.~Sylvester, J.~Steiner, and A.~Weber. Over many years a large body of research on this topic contributed to the development in various areas including integer programming, combinatorial and graph optimization (see, e.g. \cite{Elzinga72Geometrical,Hansen81Constrained,Sule01Logistics,Moradi09Single,Eiselt11Pioneering,Drezner11Continuous}).

Among other solution approaches to location problems are models and methods of idempotent algebra \cite{Baccelli92Synchronization,CuninghameGreen94Minimax,Kolokoltsov97Idempotent,Golan03Semirings,Heidergott05Maxplus,Butkovic10Maxlinear}, which find expanding applications in the analysis of actual problems in engineering, manufacturing, information technology, and other fields. Expressed in terms of idempotent algebra, a range of problems that are nonlinear in the ordinary sense, become linear and so allow more simple analysis and solution techniques. Specifically, many classical problems in graph optimization and dynamic programming reduce to solving linear vector equations, finding eigenvalues and eigenvectors of matrices, and to similar computational procedures. 

A single facility one-dimensional location problem on a graph is examined in \cite{CuninghameGreen91Minimax,CuninghameGreen94Minimax}, where it is turned into a problem of minimizing a rational function in the idempotent algebra sense. However, the proposed solution deals with polynomial and rational functions of one variable, and becomes less or no applicable in the multidimensional case.

In \cite{Zimmermann03Disjunctive,Tharwat08Oneclass}, a multidimensional constrained location problem on a graph is reduced to minimization of a max-separable objective function that can be represented as a maximum of functions each depending only on one variable. An efficient computational procedure is proposed which, however, seems to have limited application only to location problems where the objective function appears to be max-separable. 

In this paper, we further develop the algebraic approach proposed in \cite{Krivulin11Algebraic,Krivulin11Anextremal,Krivulin11Algebraicsolution}. We consider a multidimensional minimax single facility location problem with Chebyshev distance, and show how the problem can be solved based on new results in the spectral theory of matrices in idempotent algebra. The aim of the paper is twofold: first, to give a new algebraic solution to the location problem, and second, to extend the area of application of idempotent algebra.

The rest of the paper is as follows. We begin with an overview of preliminary definitions and results in idempotent algebra, including basic concepts of scalar and matrix algebra, and elements of the spectral theory of matrices. Furthermore, a new algebraic approach based on investigation of extremal properties of eigenvalues for irreducible matrices is developed to solve multidimensional problems that involve minimization of functionals defined on idempotent vector semimodules.

We examine an unconstrained minimax location problem and represent it in terms of idempotent algebra. A new solution is given that reduces the problem to evaluation of the eigenvalue and eigenvectors of an irreducible matrix. Finally, the solution is extended to solve a constrained location problem.

\section{Preliminary Results}

We start with a brief overview of definitions, notation and preliminary results of idempotent algebra that underlie the solution approach developed in subsequent sections. Further details can be found in \cite{Baccelli92Synchronization,Kolokoltsov97Idempotent,CuninghameGreen94Minimax,Golan03Semirings,Heidergott05Maxplus,Butkovic10Maxlinear}.

\subsection{Idempotent Semifield}

Let $\mathbb{X}$ be a set with two operations, addition $\oplus$ and multiplication $\otimes$, and their respective neutral elements, zero $\mathbb{0}$ and identity $\mathbb{1}$. We suppose that  $(\mathbb{X},\mathbb{0},\mathbb{1},\oplus,\otimes)$ is a commutative semiring where addition is idempotent and multiplication is invertible. Since the nonzero elements of the semiring form a group under multiplication, the semiring is usually referred to as idempotent semifield. 

The integer power is defined in the ordinary way. Let us put $\mathbb{X}_{+}=\mathbb{X}\setminus\{\mathbb{0}\}$. For any $x\in\mathbb{X}_{+}$ and integer $p>0$, we have $x^{0}=\mathbb{1}$, $\mathbb{0}^{p}=\mathbb{0}$, and
$$
x^{p}=x^{p-1}\otimes x=x\otimes x^{p-1},
\qquad
x^{-p}=(x^{-1})^{p}.
$$

We assume that the integer power can naturally be extended to the case of rational and real exponents.

In what follows, we omit, as is customary, the multiplication sign $\otimes$. The power notation is used in the sense of idempotent algebra.

The idempotent addition allows one to define a relation of partial order $\leq$ such that $x\leq y$ if and only if $x\oplus y=y$. From the definition it follows that
$$
x\leq x\oplus y,
\qquad
y\leq x\oplus y,
$$
as well as that the addition and multiplication are both isotonic. Below the relation symbols and the operator $\min$ are thought of as referring to this partial order.

It is easy to verify that the binomial identity now takes the form
$$
(x\oplus y)^{\alpha}=x^{\alpha}\oplus y^{\alpha}
$$
for all real $\alpha\geq0$.

As an example, one can consider the idempotent semifield of real numbers
$$
\mathbb{R}_{\max,+}
=
(\mathbb{R}\cup\{-\infty\},-\infty,0,\max,+).
$$

In the semifield $\mathbb{R}_{\max,+}$, there are the null and identity elements defined as $\mathbb{0}=-\infty$ and $\mathbb{1}=0$. For each $x\in\mathbb{R}$, there exists its inverse $x^{-1}$ equal to the opposite number $-x$ in conventional arithmetic. For any $x,y\in\mathbb{R}$, the power $x^{y}$ corresponds to the arithmetic product $xy$. The partial order induced by the idempotent addition coincides with the natural linear order defined on $\mathbb{R}$.

\subsection{Vectors and Matrices}

Vector and matrix operations are routinely introduced based on the scalar addition and multiplication defined on $\mathbb{X}$. Consider the Cartesian power $\mathbb{X}^{n}$ with its elements represented as column vectors. For any two vectors $\bm{x}=(x_{i})$ and $\bm{y}=(y_{i})$, and a scalar $c\in\mathbb{X}$, vector addition and multiplication by scalars follow the rules
$$
\{\bm{x}\oplus\bm{y}\}_{i}
=
x_{i}\oplus y_{i},
\qquad
\{c\bm{x}\}_{i}
=
cx_{i}.
$$ 

A vector with all zero elements is referred to as zero vector and denoted by $\mathbb{0}$.

The set $\mathbb{X}^{n}$ with these operations is a vector semimodule over the idempotent semifield $\mathbb{X}$.

A vector $\bm{y}\in\mathbb{X}^{n}$ is linearly dependent on vectors $\bm{x}_{1},\ldots,\bm{x}_{m}\in\mathbb{X}^{n}$, if there are scalars $c_{1},\ldots,c_{m}\in\mathbb{X}$ such that
$$
\bm{y}
=
c_{1}\bm{x}_{1}\oplus\cdots\oplus c_{m}\bm{x}_{m}.
$$

In particular, the vector $\bm{y}$ is collinear with $\bm{x}$, if $\bm{y}=c\bm{x}$. The zero vector is dependent on any vector.

For any column vector $\bm{x}=(x_{i})\in\mathbb{X}_{+}^{n}$, we define a row vector $\bm{x}^{-}=(x_{i}^{-})$ with its elements $x_{i}^{-}=x_{i}^{-1}$. For all $\bm{x},\bm{y}\in\mathbb{X}_{+}^{n}$, the componentwise inequality $\bm{x}\leq\bm{y}$ implies $\bm{x}^{-}\geq\bm{y}^{-}$.

For conforming matrices $A=(a_{ij})$, $B=(b_{ij})$, and $C=(c_{ij})$, matrix addition and multiplication together with  multiplication by a scalar $c\in\mathbb{X}$ are performed according to the formulas
\begin{gather*}
\{A\oplus B\}_{ij}
=
a_{ij}\oplus b_{ij},
\qquad
\{B C\}_{ij}
=
\bigoplus_{k}b_{ik}c_{kj},
\\
\{cA\}_{ij}=ca_{ij}.
\end{gather*}

A matrix with all zero entries is a zero matrix which is denoted by $\mathbb{0}$.

Consider the set of square matrices $\mathbb{X}^{n\times n}$. The matrix that has all diagonal entries equal to $\mathbb{1}$ and off-diagonal entries equal to $\mathbb{0}$ is an identity matrix denoted by $I$.

With respect to matrix addition and multiplication, the set $\mathbb{X}^{n\times n}$ forms idempotent semiring with identity.

For any matrix $A\ne\mathbb{0}$ and an integer $p>0$, the power notation is routinely defined as
$$
A^{0}=I,
\qquad
A^{p}=A^{p-1}A=AA^{p-1}.
$$

The trace of the matrix $A=(a_{ij})$ is calculated as
$$
\mathop\mathrm{tr}A=\bigoplus_{i=1}^{n}a_{ii}.
$$

A matrix is called irreducible if and only if it cannot be put in a block triangular form by simultaneous permutations of rows and columns. Otherwise the matrix is reducible.

\subsection{Eigenvalues and Eigenvectors}

A scalar $\lambda\in\mathbb{X}$ is eigenvalue of a matrix $A\in\mathbb{X}^{n\times n}$ if there exists a nonzero vector $\bm{x}\in\mathbb{X}^{n}$ such that
$$
A\bm{x}=\lambda\bm{x}.
$$
Any vector $\bm{x}\ne\mathbb{0}$ that satisfies the above equality is an eigenvector of $A$, corresponding to $\lambda$.

If the matrix $A$ is irreducible, then it has only one eigenvalue given by
\begin{equation}
\lambda
=
\bigoplus_{m=1}^{n}\mathop\mathrm{tr}\nolimits^{1/m}(A^{m}).
\label{E-lambda}
\end{equation}

The corresponding eigenvectors of $A$ have no zero entries and are found as follows. First we evaluate the matrix
$$
A^{\times}
=
\lambda^{-1}A\oplus\cdots\oplus(\lambda^{-1}A)^{n}.
$$

Let $\bm{a}_{i}^{\times}$ be column $i$ in $A^{\times}$, and $a_{ii}^{\times}$ be its diagonal element. Now each column $\bm{a}_{i}^{\times}$ is replaced with that defined as 
$$
\bm{a}_{i}^{+}
=
\begin{cases}
\bm{a}_{i}^{\times}, & \text{if $a_{ii}^{\times}=\mathbb{1}$},
\\
\mathbb{0}, & \text{otherwise}.
\end{cases}
$$

Furthermore, the set of columns $\bm{a}_{i}^{+}$ is reduced by removing those columns, if any, that are linearly dependent on others. Finally, the rest columns are put together to form a matrix $A^{+}$.

The set of all eigenvectors of $A$ corresponding to $\lambda$ (together with zero vector) coincides with the linear span of the columns of $A^{+}$, whereas each vector takes the form
$$
\bm{x}
=
A^{+}\bm{v},
$$
where $\bm{v}$ is a nonzero vector of appropriate size.

\section{Extremal Property of Eigenvalues}

Suppose $A\in\mathbb{X}^{n\times n}$ is an irreducible matrix with an eigenvalue $\lambda$. For each $\bm{x}\in\mathbb{X}_{+}^{n}$ consider a function
$$
\varphi(\bm{x})
=
\bm{x}^{-}A\bm{x}.
$$

It has been shown in \cite{Krivulin05Evaluation,Krivulin06Eigenvalues} that $\varphi(\bm{x})$ has a minimum equal to $\lambda$ and attained at any eigenvector of $A$.

Now we improve this result by extending the set of vectors that provide the minimum of $\varphi(\bm{x})$.

First we revise the above result as follows.
\begin{lemma}
\label{L-xAxlambda}
Let $A=(a_{ij})\in\mathbb{X}^{n\times n}$ be an irreducible matrix with an eigenvalue $\lambda$. Suppose $\bm{u}=(u_{i})$ and $\bm{v}=(v_{i})$ are eigenvectors of the respective  matrices $A$ and $A^{T}$. Then it holds that
\begin{equation}
\min_{\bm{x}\in\mathbb{X}_{+}^{n}}\bm{x}^{-} A\bm{x}
=
\lambda,
\label{E-xAxlambda}
\end{equation}
where the minimum is attained at $\bm{u}$ and $(\bm{v}^{-})^{T}$.
\end{lemma}
\begin{proof}
It is easy to verify that any vector $\bm{x}$ with nonzero elements satisfies the inequality $\bm{x}^{-}A\bm{x}\geq\lambda$. Indeed, let us take the eigenvector $\bm{u}$ and note that $\bm{x}\bm{u}^{-}\geq(\bm{x}^{-}\bm{u})^{-1}I$. Furthermore, we have
$$
\bm{x}^{-} A\bm{x}
=
\bm{x}^{-} A\bm{x}\bm{u}^{-}\bm{u}
\geq
\bm{x}^{-}A\bm{u}(\bm{x}^{-}\bm{u})^{-1}
=
\lambda.
$$

It remains to present particular vectors $\bm{x}$ that turn the inequality into an equality. With $\bm{x}=\bm{u}$ we have
$$
\bm{x}^{-}A\bm{x}
=
\bm{u}^{-}A\bm{u}
=
\lambda\bm{u}^{-}\bm{u}
=
\lambda.
$$

Similarly, when $\bm{x}=(\bm{v}^{-})^{T}$, we get the equality
$$
\bm{x}^{-}A\bm{x}
=
\bm{x}^{T}A^{T}(\bm{x}^{-})^{T}
=
\bm{v}^{-}A^{T}\bm{v}
=
\lambda\bm{v}^{-}\bm{v}
=
\lambda,
$$
which completes the proof.
\end{proof}

Assuming that a matrix $A=(a_{ij})\in\mathbb{X}^{n\times n}$ is irreducible and has an eigenvalue $\lambda$, we denote the set of vectors $\bm{x}$ that give minimum of $\bm{x}^{-}A\bm{x}=\lambda$ by
$$
X_{A}
=
\arg\min_{\bm{x}\in\mathbb{X}_{+}^{n}}\bm{x}^{-}A\bm{x}.
$$

Now we show that the set $X_{A}$ is closed under main operations on vectors in $\mathbb{X}^{n}$.

\begin{lemma}\label{L-cxXA}
Suppose that $\bm{x},\bm{y}\in X_{A}$ and $c\in\mathbb{X}$. Then the following statements are valid:
\begin{enumerate}
\item[(a)] $c\bm{x}\in X_{A}$;
\item[(b)] $\bm{x}\oplus\bm{y}\in X_{A}$;
\item[(c)] $(\bm{x}^{-}\oplus\bm{y}^{-})^{-}\in X_{A}$.
\end{enumerate}
\end{lemma}
\begin{proof}
The first statement is obvious. To verify the next one, we take vectors $\bm{x}=(x_{i})$ and $\bm{y}=(y_{i})$, and consider a vector $\bm{z}=(z_{i})$ defined as $\bm{z}=\bm{x}\oplus\bm{y}$. With the condition $\bm{x},\bm{y}\in X_{A}$, we have
\begin{multline*}
\lambda
=
\bm{x}^{-}A\bm{x}
\oplus
\bm{y}^{-}A\bm{y}
=
\bigoplus_{i=1}^{n}\bigoplus_{j=1}^{n}x_{i}^{-1}a_{ij}x_{j}
\oplus
\bigoplus_{i=1}^{n}\bigoplus_{j=1}^{n}y_{i}^{-1}a_{ij}y_{j}
\\
\geq
\bigoplus_{i=1}^{n}\bigoplus_{j=1}^{n}(x_{i}\oplus y_{i})^{-1}a_{ij}(x_{j}\oplus y_{j})
=
\bigoplus_{i=1}^{n}\bigoplus_{j=1}^{n}z_{i}^{-1}a_{ij}z_{j}
=
\bm{z}^{-}A\bm{z}.
\end{multline*}

Hence we arrive at the inequality $\lambda\geq\bm{z}^{-}A\bm{z}$. Since the opposite inequality is always valid, we conclude that $\lambda=\bm{z}^{-}A\bm{z}$. Therefore, $\bm{z}=\bm{x}\oplus\bm{y}\in X_{A}$.

The last statement is verified in much the same way. We put $\bm{z}=(\bm{x}^{-}\oplus\bm{y}^{-})^{-}$ and then note that
\begin{multline*}
\lambda
=
\bm{x}^{-}A\bm{x}
\oplus
\bm{y}^{-}A\bm{y}
=
\bigoplus_{i=1}^{n}\bigoplus_{j=1}^{n}x_{i}^{-1}a_{ij}x_{j}
\oplus
\bigoplus_{i=1}^{n}\bigoplus_{j=1}^{n}y_{i}^{-1}a_{ij}y_{j}
\\
\geq
\bigoplus_{i=1}^{n}\bigoplus_{j=1}^{n}(x_{i}^{-1}\oplus y_{i}^{-1})a_{ij}(x_{j}^{-1}\oplus y_{j}^{-1})^{-1}
=
\bigoplus_{i=1}^{n}\bigoplus_{j=1}^{n}z_{i}^{-1}a_{ij}z_{j}
=
\bm{z}^{-}A\bm{z}.
\end{multline*}

The rest of the proof is as before.
\end{proof}

Note that with the first and second statements of Lemma~\ref{L-cxXA}, the set $X_{A}$ appears to be a vector subsemimodule in the semimodule $\mathbb{X}$.

\begin{lemma}\label{L-xay1aXA}
Suppose vectors $\bm{x}=(x_{i})$ and $\bm{y}=(y_{i})$ satisfy the condition $\bm{x},\bm{y}\in X_{A}$. Then for all real $\alpha$ such that $0\leq\alpha\leq1$, it holds that
$$
\left(
\begin{array}{c}
x_{1}^{\alpha}y_{1}^{1-\alpha}
\\
\vdots
\\
x_{n}^{\alpha}y_{n}^{1-\alpha}
\end{array}
\right)
\in X_{A}.
$$
\end{lemma}
\begin{proof}
Assuming that $\bm{z}=(x_{1}^{\alpha}y_{1}^{1-\alpha},\ldots,x_{n}^{\alpha}y_{n}^{1-\alpha})^{T}$, where $0\leq\alpha\leq1$, we have
\begin{multline*}
\lambda
=
(\bm{x}^{-}A\bm{x})^{\alpha}
(\bm{y}^{-}A\bm{y})^{1-\alpha}
=
\bigoplus_{i=1}^{n}\bigoplus_{j=1}^{n}x_{i}^{-\alpha}a_{ij}^{\alpha}x_{j}^{\alpha}
\bigoplus_{k=1}^{n}\bigoplus_{l=1}^{n}y_{k}^{-(1-\alpha)}a_{kl}^{1-\alpha}y_{l}^{1-\alpha}
\\
\geq
\bigoplus_{i=1}^{n}\bigoplus_{j=1}^{n}x_{i}^{-\alpha}y_{i}^{-(1-\alpha)}a_{ij}x_{j}^{\alpha}y_{j}^{1-\alpha}
=
\bigoplus_{i=1}^{n}\bigoplus_{j=1}^{n}z_{i}^{-1}a_{ij}z_{j}
=
\bm{z}^{-}A\bm{z}.
\end{multline*}

Using the same arguments as in the previous lemma, we arrive at the desired result $\bm{z}\in X_{A}$.
\end{proof}

Consider a particular matrix $A$ that has the form
\begin{equation}
A
=
\left(
\begin{array}{cccc}
\mathbb{0} & a_{12} & \ldots & a_{1n}
\\
a_{21} & \mathbb{0} & \ldots & \mathbb{0}
\\
\vdots & \vdots & \ddots & \vdots
\\
a_{n1} & \mathbb{0} & \ldots & \mathbb{0}
\end{array}
\right),
\label{E-Aa2a2}
\end{equation}
where all entries $a_{12},\ldots,a_{1n}$ and $a_{21},\ldots,a_{n1}$ are assumed to be nonzero.

Denote the first column and row of $A$ as follows:
$$
\bm{a}
=
(\mathbb{0},a_{21},\ldots,a_{n1})^{T},
\qquad
\bm{b}^{-}
=
(\mathbb{0},a_{12},\ldots,a_{1n}).
$$

Note that now we can write
$$
\bm{x}^{-}A\bm{x}
=
\bm{x}^{-}\bm{a}
\oplus
\bm{b}^{-}\bm{x}.
$$

It is not difficult to see that matrix \eqref{E-Aa2a2} is irreducible. In the case of this matrix, the result of Lemma~\ref{L-xay1aXA} can be refined as follows.
\begin{lemma}
\label{L-xiaiyi1aiXA}
Suppose vectors $\bm{x}=(x_{i})$ and $\bm{y}=(y_{i})$ satisfy the condition $\bm{x},\bm{y}\in X_{A}$ for matrix \eqref{E-Aa2a2}. Then for all real $\alpha_{i}$ such that $0\leq\alpha_{i}\leq1$, it holds that
$$
\left(
\begin{array}{c}
x_{1}^{\alpha_{1}}y_{1}^{1-\alpha_{1}}
\\
\vdots
\\
x_{n}^{\alpha_{n}}y_{n}^{1-\alpha_{n}}
\end{array}
\right)
\in X_{A}.
$$
\end{lemma}
\begin{proof}
For each $i=1,\ldots,n$ we take a number $\alpha_{i}$ such that $0\leq\alpha_{i}\leq1$, and define vectors
$$
\bm{z}
=
\left(
\begin{array}{c}
x_{1}^{\alpha_{1}}y_{1}^{1-\alpha_{1}}
\\
\vdots
\\
x_{n}^{\alpha_{n}}y_{n}^{1-\alpha_{n}}
\end{array}
\right),
\quad
\bm{z}_{i}
=
\left(
\begin{array}{c}
x_{1}^{\alpha_{i}}y_{1}^{1-\alpha_{i}}
\\
\vdots
\\
x_{n}^{\alpha_{i}}y_{n}^{1-\alpha_{i}}
\end{array}
\right).
$$

It follows from Lemma~\ref{L-xay1aXA} that $\bm{z}_{1},\ldots,\bm{z}_{n}\in X_{A}$. Furthermore, we write
\begin{multline*}
\lambda
=
\bigoplus_{i=1}^{n}
\bm{z}_{i}^{-}A\bm{z}_{i}
=
\bigoplus_{i=1}^{n}
(\bm{z}_{i}^{-}\bm{a}\oplus\bm{b}^{-}\bm{z}_{i})
\\
=
\bigoplus_{i=1}^{n}
\bigoplus_{j=1}^{n}
\left(x_{j}^{-\alpha_{i}}y_{j}^{-(1-\alpha_{i})}a_{j}
\oplus
b_{j}^{-1}x_{j}^{\alpha_{i}}y_{j}^{1-\alpha_{i}}\right)
\\
\geq
\bigoplus_{i=1}^{n}
\left(x_{i}^{-\alpha_{i}}y_{i}^{-(1-\alpha_{i})}a_{i}
\oplus
b_{i}^{-1}x_{i}^{\alpha_{i}}y_{i}^{1-\alpha_{i}}\right)
=
\bm{z}^{-}\bm{a}\oplus\bm{b}^{-}\bm{z}
=
\bm{z}^{-}A\bm{z}.
\end{multline*}

The rest of the proof goes through as before.
\end{proof}

By combining the results of Lemmas~\ref{L-xAxlambda} and \ref{L-xiaiyi1aiXA}, we can arrive at the following statement.
\begin{lemma}
\label{L-uiaivai1}
Let $A$ be a matrix defined as \eqref{E-Aa2a2} with an eigenvalue $\lambda$. Suppose $\bm{u}=(u_{i})$ and $\bm{v}=(v_{i})$ are eigenvectors of the respective  matrices $A$ and $A^{T}$. Then \eqref{E-xAxlambda} is valid for any vector
$$
\bm{x}
=
\left(
\begin{array}{c}
u_{1}^{\alpha_{1}}v_{1}^{\alpha_{1}-1}
\\
\vdots
\\
u_{n}^{\alpha_{n}}v_{n}^{\alpha_{n}-1}
\end{array}
\right),
\qquad
0\leq\alpha_{1},\ldots,\alpha_{n}\leq1.
$$
\end{lemma}
\begin{proof}
It follows from Lemma~\ref{L-xAxlambda}, that $\bm{u}\in X_{A}$ and $(\bm{v}^{-})^{T}\in X_{A}$. It remains to apply Lemma~\ref{L-xiaiyi1aiXA} so as to complete the proof. 
\end{proof}

\section{Unconstrained Location Problem}

In this section we examine a minimax single facility location problem with Chebyshev distance when no constraints are imposed on the feasible location area.

Consider any two vectors $\bm{r}=(r_{1},\ldots,r_{n})^{T}$ and $\bm{s}=(s_{1},\ldots,s_{n})^{T}$ in $\mathbb{R}^{n}$. The Chebyshev distance ($L_{\infty}$ or maximum metric) is calculated as
\begin{equation}
\rho(\bm{r},\bm{s})
=
\max_{1\leq i\leq n}|r_{i}-s_{i}|.
\label{M-Chebyshev}
\end{equation}

Given $m\geq2$ vectors $\bm{r}_{i}=(r_{1i},\ldots,r_{ni})^{T}\in\mathbb{R}^{n}$ and constants $w_{i}\in\mathbb{R}$, $i=1,\ldots,m$, the location problem under examination is to determine the vectors $\bm{x}=(x_{1},\ldots,x_{n})^{T}\in\mathbb{R}^{n}$ that provide the minimum
\begin{equation}
\min_{\bm{x}\in\mathbb{R}^{n}}\max_{1\leq i\leq m}(\rho(\bm{r}_{i},\bm{x})+w_{i}).
\label{P-Chebyshev}
\end{equation}

Note that such problems are known as unweighted Rawls problems with addends \cite{Hansen81Constrained}. Following the terminology of \cite{Elzinga72Geometrical}, the problem can also be referred to as the multidimensional Chebyshev Messenger Boy Problem.

It is not difficult to solve the problem on the plane by using geometric arguments (see, eg, \cite{Sule01Logistics,Moradi09Single}). Below we give a new algebraic solution that is based on representation of the problem in terms of the idempotent semifield $\mathbb{R}_{\max,+}$, and application of the result from the previous section.

\subsection{Algebraic Representation}
First we rewrite \eqref{M-Chebyshev} as follows
$$
\rho(\bm{r},\bm{s})
=
\bm{s}^{-}\bm{r}\oplus\bm{r}^{-}\bm{s}.
$$

Denote the objective function in problem \eqref{P-Chebyshev} by $\varphi(\bm{x})$ and write
$$
\varphi(\bm{x})
=
\bigoplus_{i=1}^{m}w_{i}\rho(\bm{r}_{i},\bm{x}).
$$

With the vectors
$$
\bm{p}
=
w_{1}\bm{r}_{1}\oplus\cdots\oplus w_{m}\bm{r}_{m},
\qquad
\bm{q}^{-}
=
w_{1}\bm{r}_{1}^{-}\oplus\cdots\oplus w_{m}\bm{r}_{m}^{-},
$$
we have
$$
\varphi(\bm{x})
=
\bigoplus_{i=1}^{m}w_{i}(\bm{x}^{-}\bm{r}_{i}\oplus\bm{r}_{i}^{-}\bm{x})
=
\bm{x}^{-}\bm{p}
\oplus
\bm{q}^{-}\bm{x},
$$
and then represent problem \eqref{P-Chebyshev} as
\begin{equation}\label{P-Chebyshev1}
\min_{\bm{x}\in\mathbb{R}^{n}}\varphi(\bm{x}).
\end{equation}

Furthermore, we introduce a vector
$$
\bm{y}
=
\left(
\begin{array}{c}
y_{0}
\\
y_{1}
\\
\vdots
\\
y_{n}
\end{array}
\right)
=
\left(
\begin{array}{c}
\mathbb{1} \\
\bm{x}
\end{array}
\right),
$$
and a matrix of order $n+1$
$$
A
=
\left(
\begin{array}{cc}
\mathbb{0} & \bm{q}^{-} \\
\bm{p} & \mathbb{0}
\end{array}
\right).
$$

Since we now have
$$
\varphi(\bm{x})=\bm{x}^{-}\bm{p}\oplus\bm{q}^{-}\bm{x}=\bm{y}^{-}A\bm{y},
$$
problem \eqref{P-Chebyshev1} reduces to that of the form
\begin{equation}\label{P-Chebyshev2}
\min_{\bm{y}\in\mathbb{R}^{n+1}}\bm{y}^{-}A\bm{y}.
\end{equation}

Note that the vectors $\bm{y}\in\mathbb{R}^{n+1}$ that solve \eqref{P-Chebyshev2} do not always have an appropriate form to give a solution to \eqref{P-Chebyshev1}. Specifically, to be consistent to \eqref{P-Chebyshev1}, the vector $\bm{y}$ must have the first element equal to $\mathbb{1}$.

\subsection{Algebraic Solution}

Consider problem \eqref{P-Chebyshev2}, and note that the matrix $A$ has the form of \eqref{E-Aa2a2} and it is irreducible. It follows from Lemma~\ref{L-uiaivai1} that
$$
\min_{\bm{y}\in\mathbb{R}^{n+1}}\bm{y}^{-}A\bm{y}
=
\lambda,
$$
where $\lambda$ is the eigenvector of $A$, and the minimum is attained at a vector that is obtained from eigenvectors $\bm{u}=(u_{i})$ and $\bm{v}=(v_{i})$ of matrices $A$ and $A^{T}$.

First we evaluate $\lambda$. For all $k=1,2,\ldots$ we have
$$
A^{2k-1}
=
(\bm{q}^{-}\bm{p})^{k-1}
\left(
\begin{array}{cc}
\mathbb{0} & \bm{q}^{-} \\
\bm{p} & \mathbb{0}
\end{array}
\right),
\qquad
A^{2k}
=
(\bm{q}^{-}\bm{p})^{k-1}
\left(
\begin{array}{cc}
\bm{q}^{-}\bm{p} & \mathbb{0} \\
\mathbb{0} & \bm{p}\bm{q}^{-}
\end{array}
\right),
$$
and therefore,
$$
\mathop\mathrm{tr}(A^{2k-1})
=
\mathbb{0},
\qquad
\mathop\mathrm{tr}(A^{2k})
=
(\bm{q}^{-}\bm{p})^{k}.
$$

Finally, application of \eqref{E-lambda} gives
$$
\lambda
=
\bigoplus_{m=1}^{n}\mathop\mathrm{tr}\nolimits^{1/m}(A^{m})
=
(\bm{q}^{-}\bm{p})^{1/2}.
$$

To get vectors that produce the minimum in \eqref{P-Chebyshev2}, we need to derive the eigenvectors of the matrices $A$ and $A^{T}$. Note that $A^{T}$ is obtained from $A$ by replacement of $\bm{p}$ with $(\bm{q}^{-})^{T}$ and $\bm{q}^{-}$ with $\bm{p}^{T}$. Therefore, it will suffice to find the eigenvectors for $A$, and then turn them into those for $A^{T}$ by the above replacement.

To obtain the eigenvectors of $A$, we consider the matrix
$$
\lambda^{-1}A
=
(\bm{q}^{-}\bm{p})^{-1/2}
\left(
\begin{array}{cc}
\mathbb{0} & \bm{q}^{-} \\
\bm{p} & \mathbb{0}
\end{array}
\right).
$$

Since for any $k=1,2,\ldots$ it holds that
\begin{align*}
(\lambda^{-1}A)^{2k-1}
&=
(\bm{q}^{-}\bm{p})^{-1/2}
\left(
\begin{array}{cc}
\mathbb{0} & \bm{q}^{-} \\
\bm{p} & \mathbb{0}
\end{array}
\right),
\\
(\lambda^{-1}A)^{2k}
&=
(\bm{q}^{-}\bm{p})^{-1}
\left(
\begin{array}{cc}
\bm{q}^{-}\bm{p} & \mathbb{0} \\
\mathbb{0} & \bm{p}\bm{q}^{-}
\end{array}
\right),
\end{align*}
we arrive at the matrix $A^{\times}$ in the form
$$
A^{\times}
=
\lambda^{-1}A\oplus\cdots\oplus(\lambda^{-1}A)^{n+1}
=
\left(
\begin{array}{cc}
\mathbb{1} & (\bm{q}^{-}\bm{p})^{-1/2}\bm{q}^{-} \\
(\bm{q}^{-}\bm{p})^{-1/2}\bm{p} & (\bm{q}^{-}\bm{p})^{-1}\bm{p}\bm{q}^{-}
\end{array}
\right).
$$

It is not difficult to verify that in the matrix $A^{\times}$, any column that has $\mathbb{1}$ on the diagonal is collinear with the first column. Indeed, suppose that the submatrix $(\bm{q}^{-}\bm{p})^{-1}\bm{p}\bm{q}^{-}$ has a diagonal element equal to $\mathbb{1}$, say the element in its first column (that corresponds to the second column of $A^{\times}$). In this case, we have $\bm{q}^{-}\bm{p}=q_{1}^{-1}p_{1}$. The matrix $A^{\times}$ takes the form
$$
A^{\times}
=
\left(
\begin{array}{cc}
\mathbb{1} & q_{1}^{1/2}p_{1}^{-1/2}\bm{q}^{-} \\
q_{1}^{1/2}p_{1}^{-1/2}\bm{p} & q_{1}p_{1}^{-1}\bm{p}\bm{q}^{-}
\end{array}
\right)
=
\left(
\begin{array}{ccc}
\mathbb{1} & q_{1}^{-1/2}p_{1}^{-1/2} & \ldots\\
q_{1}^{1/2}p_{1}^{-1/2}\bm{p} & p_{1}^{-1}\bm{p} & \ldots
\end{array}
\right),
$$
where the second column obviously proves to be collinear with the first one.

Let us construct a matrix $A^{+}$ that includes such columns of $A^{\times}$ that have the diagonal element equal to $\mathbb{1}$ and are independent on each other. Since all the columns with $\mathbb{1}$ on the diagonal are collinear with the first one, they can be omitted.

With the matrix $A^{+}$ formed from the first column of $A^{\times}$, we finally represent any eigenvector of $A$ as
$$
\bm{u}
=
\left(
\begin{array}{c}
\mathbb{1} \\
(\bm{q}^{-}\bm{p})^{-1/2}\bm{p}
\end{array}
\right)s,
\qquad
s\in\mathbb{R}.
$$
 
By replacing $\bm{p}$ with $(\bm{q}^{-})^{T}$ and $\bm{q}^{-}$ with $\bm{p}^{T}$, we get the eigenvectors of $A^{T}$
$$
\bm{v}
=
\left(
\begin{array}{c}
\mathbb{1} \\
(\bm{q}^{-}\bm{p})^{-1/2}(\bm{q}^{-})^{T}
\end{array}
\right)t,
\qquad
t\in\mathbb{R}.
$$

Application of Lemma~\ref{L-uiaivai1} gives a solution of \eqref{P-Chebyshev2} in the form
$$
\bm{y}
=
\left(
\begin{array}{c}
s^{\alpha_{0}}t^{1-\alpha_{0}}
\\
(\bm{q}^{-}\bm{p})^{1/2-\alpha_{1}}(p_{1}s)^{\alpha_{1}}(q_{1}t)^{1-\alpha_{1}}
\\
\vdots
\\
(\bm{q}^{-}\bm{p})^{1/2-\alpha_{n}}(p_{n}s)^{\alpha_{n}}(q_{n}t)^{1-\alpha_{n}}
\end{array}
\right),
\quad
s,t\in\mathbb{R},
\quad
0\leq\alpha_{0},\ldots,\alpha_{n}\leq1.
$$

With the condition that the first element of $\bm{y}$ must be equal to $\mathbb{1}$, we have to ensure the equation
$$
s^{\alpha_{0}}t^{\alpha_{0}-1}
=
\mathbb{1}
$$
to be valid for all $\alpha_{0}$ such that $0\leq\alpha_{0}\leq1$. Since the only solution to the equation is $s=t=\mathbb{1}$, we arrive at the solution of \eqref{P-Chebyshev1} given by
$$
\bm{x}
=
\left(
\begin{array}{c}
(\bm{q}^{-}\bm{p})^{1/2-\alpha_{1}}p_{1}^{\alpha_{1}}q_{1}^{1-\alpha_{1}}
\\
\vdots
\\
(\bm{q}^{-}\bm{p})^{1/2-\alpha_{n}}p_{n}^{\alpha_{n}}q_{n}^{1-\alpha_{n}}
\end{array}
\right),
\qquad
0\leq\alpha_{1},\ldots,\alpha_{n}\leq1.
$$

\subsection{Summary of Results}

We summarize the above results in the form of the following statements.

\begin{lemma}
\label{L-Chebyshev}
Suppose that $\bm{p}=(p_{i})$ and $\bm{q}=(q_{i})$ are vectors such that
$$
\bm{p}
=
w_{1}\bm{r}_{1}\oplus\cdots\oplus w_{m}\bm{r}_{m},
\qquad
\bm{q}^{-}
=
w_{1}\bm{r}_{1}^{-}\oplus\cdots\oplus w_{m}\bm{r}_{m}^{-}.
$$

Then the minimum in problem \eqref{P-Chebyshev1} is given by
$$
\lambda=(\bm{q}^{-}\bm{p})^{1/2},
$$
and it is attained at the vector
$$
\bm{x}
=
\left(
\begin{array}{c}
\lambda^{1-2\alpha_{1}}p_{1}^{\alpha_{1}}q_{1}^{1-\alpha_{1}}
\\
\vdots
\\
\lambda^{1-2\alpha_{n}}p_{n}^{\alpha_{n}}q_{n}^{1-\alpha_{n}}
\end{array}
\right)
$$
for all $\alpha_{i}$ such that $0\leq\alpha_{i}\leq1$, $i=1,\ldots,n$.
\end{lemma}

With the usual notation, we can reformulate the statement of Lemma~\ref{L-Chebyshev} as follows.
\begin{corollary}
Suppose that for each $i=1,\ldots,n$
$$
p_{i}
=
\max(r_{i1}+w_{1},\ldots,r_{im}+w_{m}),
\qquad
q_{i}
=
\min(r_{i1}-w_{1},\ldots,r_{im}-w_{m}).
$$

Then the minimum in \eqref{P-Chebyshev} is given by
$$
\lambda=\max(p_{1}-q_{1},\ldots,p_{n}-q_{n})/2,
$$
and it is attained at the vector
$$
\bm{x}
=
\left(
\begin{array}{c}
\alpha_{1}(p_{1}-\lambda)
\\
\vdots
\\
\alpha_{n}(p_{n}-\lambda)
\end{array}
\right)
+
\left(
\begin{array}{c}
(1-\alpha_{1})(q_{1}+\lambda)
\\
\vdots
\\
(1-\alpha_{n})(q_{n}+\lambda)
\end{array}
\right)
$$
for all $\alpha_{i}$ such that $0\leq\alpha_{i}\leq1$, $i=1,\ldots,n$.
\end{corollary}

An illustration of the solution is demonstrated in Fig.~\ref{F-Chebyshev1}--\ref{F-Chebyshev3}. We start with two examples in Fig.~\ref{F-Chebyshev1} that present solutions in the plane $\mathbb{R}^{2}$ when $w_{i}=0$ for all $i=1,\ldots,m$. In both examples, the given points are shown with thick dots, whereas the solution set is shown with a thick line segment.
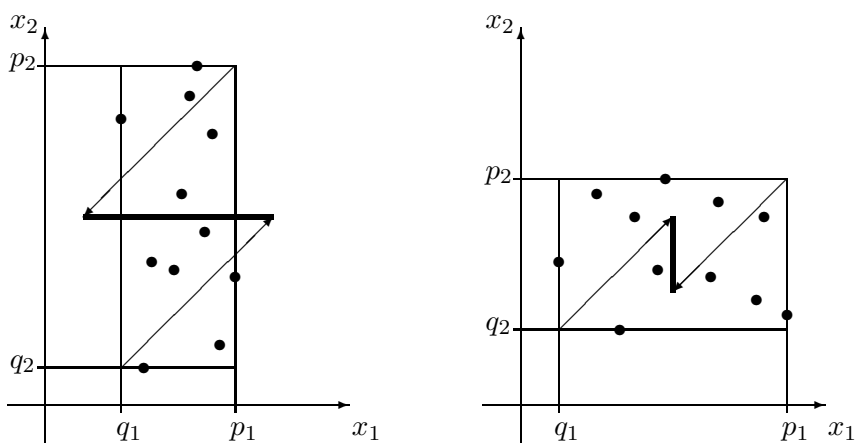
\begin{figure}[ht]
\setlength{\unitlength}{1mm}
\begin{center}

\begin{picture}(50,60)

\put(0,5){\vector(1,0){45}}
\put(5,0){\vector(0,1){55}}

\put(15,50){\line(0,-1){46}}
\put(30,50){\line(0,-1){46}}

\put(30,10){\line(-1,0){26}}
\put(30,50){\line(-1,0){26}}

\put(15,10){\vector(1,1){20}}

\put(30,50){\vector(-1,-1){20}}

\put(10,30){\linethickness{2pt}\line(1,0){25}}

\put(15,43){\circle*{1.5}}
\put(18,10){\circle*{1.5}}
\put(19,24){\circle*{1.5}}
\put(23,33){\circle*{1.5}}
\put(24,46){\circle*{1.5}}
\put(25,50){\circle*{1.5}}
\put(27,41){\circle*{1.5}}
\put(28,13){\circle*{1.5}}
\put(30,22){\circle*{1.5}}
\put(26,28){\circle*{1.5}}
\put(22,23){\circle*{1.5}}

\put(14,1){$q_{1}$}
\put(29,1){$p_{1}$}

\put(45,1){$x_{1}$}

\put(0,50){$p_{2}$}
\put(0,10){$q_{2}$}

\put(0,55){$x_{2}$}

\end{picture}
\hspace{10\unitlength}
\begin{picture}(50,60)

\put(0,5){\vector(1,0){45}}
\put(5,0){\vector(0,1){55}}

\put(10,35){\line(0,-1){31}}
\put(40,35){\line(0,-1){31}}

\put(40,15){\line(-1,0){36}}
\put(40,35){\line(-1,0){36}}

\put(10,15){\vector(1,1){15}}

\put(40,35){\vector(-1,-1){15}}

\put(25,20){\linethickness{2pt}\line(0,1){10}}

\put(15,33){\circle*{1.5}}
\put(18,15){\circle*{1.5}}
\put(10,24){\circle*{1.5}}
\put(23,23){\circle*{1.5}}
\put(24,35){\circle*{1.5}}
\put(31,32){\circle*{1.5}}
\put(20,30){\circle*{1.5}}
\put(40,17){\circle*{1.5}}
\put(30,22){\circle*{1.5}}
\put(37,30){\circle*{1.5}}
\put(36,19){\circle*{1.5}}

\put(9,1){$q_{1}$}
\put(39,1){$p_{1}$}

\put(45,1){$x_{1}$}

\put(0,35){$p_{2}$}
\put(0,15){$q_{2}$}

\put(0,55){$x_{2}$}

\end{picture}

\end{center}
\caption{Solutions in $\mathbb{R}^{2}$ when all $w_{i}=0$.}\label{F-Chebyshev1}
\end{figure}

In geometric terms, the solution is obtained as follows. Construct a minimal upright rectangle enclosing all given points. Then trace two lines that are oriented at a $45^{\circ}$ angle to the horizontal axis and go through the lower left and upper right vertices of the rectangle. The solution is the inner segment that these lines are cut off from the line drawn across the rectangle through the center points of its long sides.

Examples of solution in the space $\mathbb{R}^{3}$ are given in Fig.~\ref{F-Chebyshev2}, where the solution sets take the form of rectangles depicted by thick lines. 
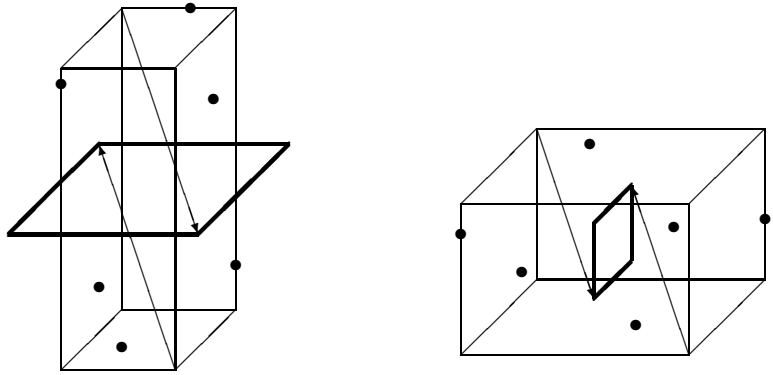
\begin{figure}[ht]
\setlength{\unitlength}{1mm}
\begin{center}

\begin{picture}(50,50)

\newsavebox\polytope
\savebox{\polytope}(30,50)[b]
{
\put(0,0){\line(1,0){15}}
\put(0,0){\line(0,1){40}}
\put(0,0){\line(1,1){8}}

\put(0,40){\line(1,0){15}}
\put(0,40){\line(1,1){8}}

\put(8,8){\line(1,0){15}}

\put(8,48){\line(1,0){15}}
\put(8,8){\line(0,1){40}}

\put(15,0){\line(0,1){40}}
\put(23,8){\line(0,1){40}}

\put(15,40){\line(1,1){8}}
\put(15,0){\line(1,1){8}}

\put(8,48){\vector(1,-3){10}}
\put(15,0){\vector(-1,3){10}}

\put(5,30.2){\thicklines\line(-1,-1){12}}
\put(5,30.0){\thicklines\line(-1,-1){12}}
\put(5,29.8){\thicklines\line(-1,-1){12}}

\put(5,30.1){\thicklines\line(1,0){25}}
\put(5,29.9){\thicklines\line(1,0){25}}

\put(18,17.8){\thicklines\line(1,1){12}}
\put(18,18.0){\thicklines\line(1,1){12}}
\put(18,18.2){\thicklines\line(1,1){12}}

\put(18,18.1){\thicklines\line(-1,0){25}}
\put(18,17.9){\thicklines\line(-1,0){25}}

\put(17,48){\circle*{1.5}}
\put(0,38){\circle*{1.5}}
\put(20,36){\circle*{1.5}}
\put(23,14){\circle*{1.5}}
\put(5,11){\circle*{1.5}}
\put(8,3){\circle*{1.5}}

}

\put(0,-2){\makebox(30,50){\usebox{\polytope}}}

\end{picture}
\hspace{10\unitlength}
\begin{picture}(50,40)

\savebox{\polytope}(40,30)[b]
{
\put(0,0){\line(1,0){30}}
\put(0,0){\line(0,1){20}}
\put(0,0){\line(1,1){10}}

\put(0,20){\line(1,0){30}}
\put(0,20){\line(1,1){10}}

\put(10,10){\line(1,0){30}}

\put(10,30){\line(1,0){30}}
\put(10,10){\line(0,1){20}}

\put(30,0){\line(0,1){20}}
\put(40,10){\line(0,1){20}}

\put(30,20){\line(1,1){10}}
\put(30,0){\line(1,1){10}}

\put(10,30){\vector(1,-3){7.5}}
\put(30,0){\vector(-1,3){7.5}}

\put(22.5,22.7){\thicklines\line(-1,-1){5}}
\put(22.5,22.5){\thicklines\line(-1,-1){5}}
\put(22.5,22.3){\thicklines\line(-1,-1){5}}

\put(22.4,22.5){\thicklines\line(0,-1){10}}
\put(22.6,22.5){\thicklines\line(0,-1){10}}

\put(17.5,7.7){\thicklines\line(1,1){5}}
\put(17.5,7.5){\thicklines\line(1,1){5}}
\put(17.5,7.3){\thicklines\line(1,1){5}}

\put(17.4,7.5){\thicklines\line(0,1){10}}
\put(17.6,7.5){\thicklines\line(0,1){10}}

\put(17,28){\circle*{1.5}}
\put(40,18){\circle*{1.5}}
\put(0,16){\circle*{1.5}}
\put(23,4){\circle*{1.5}}
\put(8,11){\circle*{1.5}}
\put(28,17){\circle*{1.5}}

}

\put(-15,0){\makebox(40,30){\usebox{\polytope}}}

\end{picture}
\end{center}
\caption{Solutions in $\mathbb{R}^{3}$ when all $w_{i}=0$.}\label{F-Chebyshev2}
\end{figure}

Fig.~\ref{F-Chebyshev3} illustrates the solution of a problem with arbitrary constants $w_{i}$. First we present solution to an auxiliary problem obtained from the initial problem by setting $w_{i}=0$ for all $i$ (top picture). To get solution in the case of nonzero constants $w_{i}$, we replace each given point with two new points. Furthermore, the minimal rectangle is built for the new points and then the solution is derived in the same way as above (bottom picture).
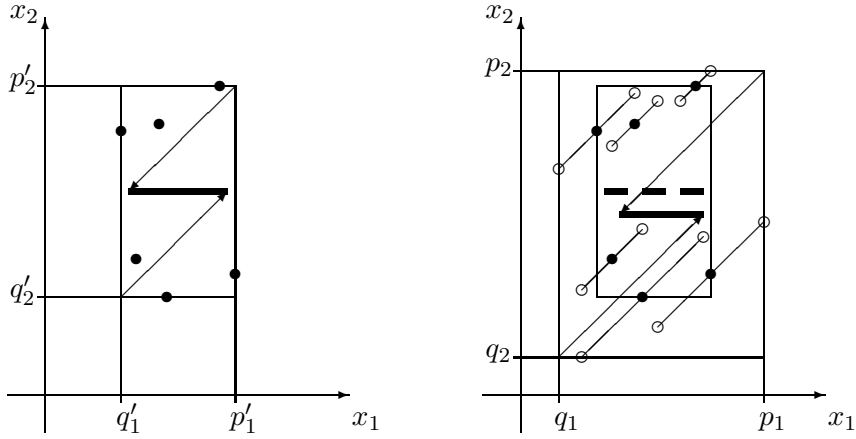
\begin{figure}[ht]
\setlength{\unitlength}{1mm}
\begin{center}

\begin{picture}(50,60)

\put(0,5){\vector(1,0){45}}
\put(5,0){\vector(0,1){55}}

\put(15,46){\line(0,-1){42}}
\put(30,46){\line(0,-1){42}}

\put(30,18){\line(-1,0){26}}
\put(30,46){\line(-1,0){26}}

\put(15,18){\vector(1,1){14}}

\put(30,46){\vector(-1,-1){14}}

\put(16,32){\linethickness{2pt}\line(1,0){13}}

\put(17,23){\circle*{1.5}}

\put(20,41){\circle*{1.5}}

\put(28,46){\circle*{1.5}}

\put(15,40){\circle*{1.5}}

\put(30,21){\circle*{1.5}}

\put(21,18){\circle*{1.5}}

\put(14,1){$q_{1}^{\prime}$}
\put(29,1){$p_{1}^{\prime}$}

\put(45,1){$x_{1}$}

\put(0,46){$p_{2}^{\prime}$}
\put(0,18){$q_{2}^{\prime}$}

\put(0,55){$x_{2}$}

\end{picture}
\hspace{10\unitlength}
\begin{picture}(50,60)

\put(0,5){\vector(1,0){45}}
\put(5,0){\vector(0,1){55}}

\put(15,46){\line(0,-1){28}}
\put(30,46){\line(0,-1){28}}

\put(30,18){\line(-1,0){15}}
\put(30,46){\line(-1,0){15}}

\multiput(16,32)(5,0){3}{\linethickness{2pt}\line(1,0){3}}

\put(10,48){\line(0,-1){44}}
\put(37,48){\line(0,-1){44}}

\put(37,10){\line(-1,0){33}}
\put(37,48){\line(-1,0){33}}

\put(10,10){\vector(1,1){19}}

\put(37,48){\vector(-1,-1){19}}

\put(18,29){\linethickness{2pt}\line(1,0){11}}

\put(17,23){\circle*{1.5}}
\put(17,23){\line(1,1){4}}
\put(21,27){\circle{1.5}}
\put(17,23){\line(-1,-1){4}}
\put(13,19){\circle{1.5}}

\put(20,41){\circle*{1.5}}
\put(17,38){\line(1,1){6}}
\put(23,44){\circle{1.5}}
\put(17,38){\circle{1.5}}

\put(28,46){\circle*{1.5}}
\put(26,44){\line(1,1){4}}
\put(30,48){\circle{1.5}}
\put(26,44){\circle{1.5}}

\put(15,40){\circle*{1.5}}
\put(15,40){\line(1,1){5}}
\put(20,45){\circle{1.5}}
\put(15,40){\line(-1,-1){5}}
\put(10,35){\circle{1.5}}

\put(30,21){\circle*{1.5}}
\put(30,21){\line(1,1){7}}
\put(37,28){\circle{1.5}}
\put(30,21){\line(-1,-1){7}}
\put(23,14){\circle{1.5}}

\put(21,18){\circle*{1.5}}
\put(21,18){\line(1,1){8}}
\put(29,26){\circle{1.5}}
\put(21,18){\line(-1,-1){8}}
\put(13,10){\circle{1.5}}

\put(9,1){$q_{1}$}
\put(36,1){$p_{1}$}

\put(45,1){$x_{1}$}

\put(0,48){$p_{2}$}
\put(0,10){$q_{2}$}

\put(0,55){$x_{2}$}

\end{picture}
\end{center}
\caption{Solution to a problem with nonzero $w_{i}$.}\label{F-Chebyshev3}
\end{figure}

\section{Constrained Location Problems}

Suppose that there is a set $S\in\mathbb{R}^{n}$ given to specify a feasible location area in problem \eqref{P-Chebyshev} and consider the constrained problem
\begin{equation}
\min_{\bm{x}\in S}\max_{1\leq i\leq m}(\rho(\bm{r}_{i},\bm{x})+w_{i}).
\label{P-ChebyshevConstrained}
\end{equation}

Representation in terms of the semifield $\mathbb{R}_{\max,+}$ leads to the problem
\begin{equation}
\min_{\bm{x}\in S}\bigoplus_{i=1}^{m}w_{i}\rho(\bm{r}_{i},\bm{x}).
\label{P-ChebyshevConstrained1}
\end{equation}

To solve the last problem we put it in the form of \eqref{P-Chebyshev1} by including the area constraints into the objective function of a normalized unconstrained problem.

\subsection{A Normalized Problem}
First, problem \eqref{P-Chebyshev1} is transformed into a normalized form to enable subsequent accommodation of the constraints in a natural way. We introduce new notation with a subscript
$$
\bm{p}_{0}
=
w_{1}\bm{r}_{1}\oplus\cdots\oplus w_{m}\bm{r}_{m},
\quad
\bm{q}_{0}^{-}
=
w_{1}\bm{r}_{1}^{-}\oplus\cdots\oplus w_{m}\bm{r}_{m}^{-},
\quad
\lambda_{0}
=
(\bm{q}_{0}^{-}\bm{p}_{0})^{1/2},
$$
and then define a normalized objective function
$$
\varphi_{0}(\bm{x})
=
\lambda_{0}^{-1}(\bm{x}^{-}\bm{p}_{0}\oplus\bm{q}_{0}^{-}\bm{x}).
$$

Instead of problem \eqref{P-Chebyshev1}, we consider the problem
$$
\min_{\bm{x}\in\mathbb{R}^{n}}\varphi_{0}(\bm{x}).
$$

It follows from Lemma~\ref{L-Chebyshev} that the normalized problem has its minimum equal to $\mathbb{1}=0$, whereas its solution set obviously coincides with that of \eqref{P-Chebyshev1}.

\subsection{Maximum Distance Constraints}

Suppose that there are constraints imposed on the maximum Chebyshev distance from the facility location point to each given points. The constraints determine the feasible location set in the form
$$
S
=
\{\bm{x}\in\mathbb{R}^{n}|\rho(\bm{r}_{i},\bm{x})\leq d_{i}, i=1,\ldots,m\}.
$$

For each $i=1,\ldots,m$, the inequality
$$
\bm{x}^{-}\bm{r}_{i}\oplus\bm{r}_{i}^{-}\bm{x}=\rho(\bm{r}_{i},\bm{x})\leq d_{i}
$$
can be rewritten in an equivalent form as
$$
d_{i}^{-1}\bm{x}^{-}\bm{r}_{i}\oplus d_{i}^{-1}\bm{r}_{i}^{-}\bm{x}\leq\mathbb{1}.
$$

With the notation
$$
\bm{p}_{1}
=
d_{1}^{-1}\bm{r}_{1}\oplus\cdots\oplus d_{m}^{-1}\bm{r}_{m},
\qquad
\bm{q}_{1}^{-}
=
d_{1}^{-1}\bm{r}_{1}^{-}\oplus\cdots\oplus d_{m}^{-1}\bm{r}_{m}^{-},
$$
all constraints are replaced with one inequality
$$
\bm{x}^{-}\bm{p}_{1}
\oplus
\bm{q}_{1}^{-}\bm{x}
\leq
\mathbb{1}.
$$

Furthermore, we introduce a function
$$
\varphi_{1}(\bm{x})
=
\bm{x}^{-}\bm{p}_{1}\oplus\bm{q}_{1}^{-}\bm{x}
$$
and note that $\varphi_{1}(\bm{x})\leq\mathbb{1}$ if and only if the maximum distance constraints are satisfied.

Finally, we put
$$
\bm{p}
=
\lambda_{0}^{-1}\bm{p}_{0}\oplus\bm{p}_{1},
\qquad
\bm{q}^{-}
=
\lambda_{0}^{-1}\bm{q}_{0}^{-}\oplus\bm{q}_{1}^{-},
$$
and define the objective function
$$
\varphi(\bm{x})
=
\varphi_{0}(\bm{x})
\oplus
\varphi_{1}(\bm{x})
=
\bm{x}^{-}\bm{p}\oplus\bm{q}^{-}\bm{x}.
$$

Now we can replace problem \eqref{P-ChebyshevConstrained1} by an unconstrained problem that has the form of problem \eqref{P-Chebyshev1} where the objective function $\varphi(\bm{x})$ is defined as above. It is clear that both problems give the same solution set provided that the solution of the unconstrained problem has nonempty intersection with the feasible set. At the same time, the new problem allows one to get approximate solutions in the case when the intersection is empty.

Based on the results offered by Lemma~\ref{L-Chebyshev}, we can give a solution to the problem under the maximum distance constraints in the following form.

\begin{lemma}\label{L-ChebyshevConstrained}
Suppose that $\bm{p}=(p_{i})$ and $\bm{q}=(q_{i})$ are vectors such that
$$
\bm{p}
=
\lambda_{0}^{-1}\bm{p}_{0}\oplus\bm{p}_{1},
\qquad
\bm{q}^{-}
=
\lambda_{0}^{-1}\bm{q}_{0}^{-}\oplus\bm{q}_{1}^{-},
$$
where
$$
\bm{p}_{0}
=
w_{1}\bm{r}_{1}\oplus\cdots\oplus w_{m}\bm{r}_{m},
\quad
\bm{q}_{0}^{-}
=
w_{1}\bm{r}_{1}^{-}\oplus\cdots\oplus w_{m}\bm{r}_{m}^{-},
\quad
\lambda_{0}
=
(\bm{q}_{0}^{-}\bm{p}_{0})^{1/2},
$$
and
$$
\bm{p}_{1}
=
d_{1}^{-1}\bm{r}_{1}\oplus\cdots\oplus d_{m}^{-1}\bm{r}_{m},
\qquad
\bm{q}_{1}^{-}
=
d_{1}^{-1}\bm{r}_{1}^{-}\oplus\cdots\oplus d_{m}^{-1}\bm{r}_{m}^{-}.
$$

Then the minimum in problem \eqref{P-ChebyshevConstrained1} is given by
$$
\lambda=(\bm{q}^{-}\bm{p})^{1/2},
$$
and it is attained at the vector
$$
\bm{x}
=
\left(
\begin{array}{c}
\lambda^{1-2\alpha_{1}}p_{1}^{\alpha_{1}}q_{1}^{1-\alpha_{1}}
\\
\vdots
\\
\lambda^{1-2\alpha_{n}}p_{n}^{\alpha_{n}}q_{n}^{1-\alpha_{n}}
\end{array}
\right)
$$
for all $\alpha_{i}$ such that $0\leq\alpha_{i}\leq1$, $i=1,\ldots,n$.
\end{lemma}

Going back to the usual notation, we arrive at the following result.
\begin{corollary}
Suppose that for each $i=1,\ldots,n$
$$
p_{i}
=
\max(p_{0i}-\lambda_{0},p_{1i}),
\qquad
q_{i}
=
\min(q_{0i}+\lambda_{0},q_{1i}),
$$
where
\begin{align*}
p_{0i}
&=
\max(r_{i1}+w_{1},\ldots,r_{im}+w_{m}),
\\
q_{0i}
&=
\min(r_{i1}-w_{1},\ldots,r_{im}-w_{m}),
\\
p_{1i}
&=
\max(r_{i1}-d_{1},\ldots,r_{im}-d_{m}),
\\
q_{1i}
&=
\min(r_{i1}+d_{1},\ldots,r_{im}+d_{m}),
\end{align*}
and
$$
\lambda_{0}
=
\max(p_{01}-q_{01},\ldots,p_{0n}-q_{0n})/2.
$$

Then the minimum in \eqref{P-ChebyshevConstrained1} is given by
$$
\lambda=\max(p_{1}-q_{1},\ldots,p_{n}-q_{n})/2,
$$
and it is attained at the vector
$$
\bm{x}
=
\left(
\begin{array}{c}
\alpha_{1}(p_{1}-\lambda)
\\
\vdots
\\
\alpha_{n}(p_{n}-\lambda)
\end{array}
\right)
+
\left(
\begin{array}{c}
(1-\alpha_{1})(q_{1}+\lambda)
\\
\vdots
\\
(1-\alpha_{n})(q_{n}+\lambda)
\end{array}
\right)
$$
for all $\alpha_{i}$ such that $0\leq\alpha_{i}\leq1$, $i=1,\ldots,n$.
\end{corollary}

Fig.~\ref{F-ChebyshevConstrained1} gives an example of solution to a problem with maximum distance constraints in $\mathbb{R}^{2}$. The entire thick line segment represents the solution of the corresponding unconstrained problem, whereas the part of the segment inside the inner rectangle indicates the solution of the constrained problem.
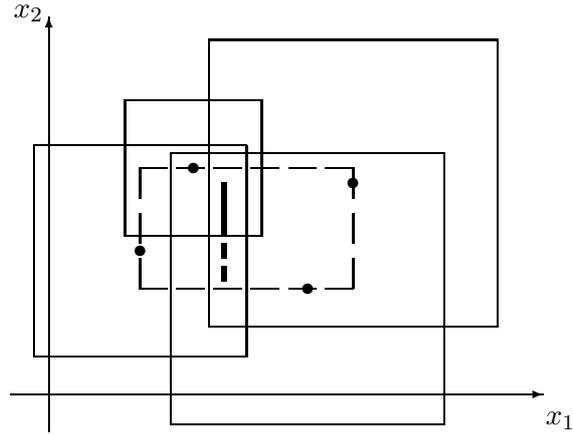
\begin{figure}[ht]
\setlength{\unitlength}{1mm}
\begin{center}

\begin{picture}(75,60)

\put(0,5){\vector(1,0){70}}
\put(5,0){\vector(0,1){55}}

\multiput(45,19)(-4.9,0){6}{\line(-1,0){3.6}}

\multiput(45,35)(-4.9,0){6}{\line(-1,0){3.6}}

\multiput(17,35)(0,-6){3}{\line(0,-1){4.0}}

\multiput(45,35)(0,-6){3}{\line(0,-1){4.0}}





\put(28,26){\linethickness{2pt}\line(0,1){7}}

\multiput(28,20)(0,3){2}{\linethickness{2pt}\line(0,1){2}}



\put(45,33){\circle*{1.5}}
\put(26,14){\line(1,0){38}}
\put(26,52){\line(1,0){38}}
\put(26,14){\line(0,1){38}}
\put(64,14){\line(0,1){38}}

\put(39,19){\circle*{1.5}}
\put(21,1){\line(1,0){36}}
\put(21,37){\line(1,0){36}}
\put(21,1){\line(0,1){36}}
\put(57,1){\line(0,1){36}}

\put(24,35){\circle*{1.5}}
\put(15,26){\line(1,0){18}}
\put(15,44){\line(1,0){18}}
\put(15,26){\line(0,1){18}}
\put(33,26){\line(0,1){18}}

\put(17,24){\circle*{1.5}}
\put(3,10){\line(1,0){28}}
\put(3,38){\line(1,0){28}}
\put(3,10){\line(0,1){28}}
\put(31,10){\line(0,1){28}}


\put(70,1){$x_{1}$}


\put(0,55){$x_{2}$}

\end{picture}
\end{center}
\caption{Solution to a constrained problem.}\label{F-ChebyshevConstrained1}
\end{figure}

\subsection*{Acknowledgments}
The work was partially supported by the Russian Foundation for Basic Research, Grant \#09-01-00808.

\bibliographystyle{utphys}

\bibliography{An_algebraic_approach_to_multidimensional_minimax_location_problems_with_Chebyshev_distance}

\end{document}